\documentclass[11pt,twoside,a4paper]{amsart}

\usepackage{amsfonts,amsmath,amsthm}
\usepackage{hyperref}

\newtheorem{thm}{Theorem}[section]

\newtheorem{lem}[thm]{Lemma}
\newtheorem{pro}[thm]{Proposition}
\newtheorem{cor}[thm]{Corollary}

\theoremstyle{remark}
\newtheorem{remark}[thm]{Remark}

\theoremstyle{ass}

\newcommand{\R}{\mathbb{R}}
\newcommand{\Rd}{\mathbb{R}^{d}}
\newcommand{\D}{\nabla_{b}}

\newcommand{\p}{\tilde{p}}

\numberwithin{equation}{section}

\begin{document}

\title[Energy concentration and explicit Sommerfeld condition]{Energy concentration and explicit Sommerfeld radiation condition for the electromagnetic Helmholtz equation}

\author{Miren Zubeldia}

\address{M. Zubeldia: Universidad del Pa\'is Vasco, Departamento de Matem\'aticas, Apartado 644, 48080, Spain}
\email{miren.zubeldia@ehu.es}

\thanks{The author is partially supported by the Spanish grant FPU AP2007-02659 of the MEC, by Spanish Grant MTM2007-62186 and by the ERC Advanced Grant-Mathematical foundations (ERC-AG-PE1) of the European Research Council}

\begin{abstract}
We study the electromagnetic Helmholtz equation
\begin{equation}\notag
(\nabla + ib(x))^{2}u(x) + n(x)u(x) = f(x), \quad x\in\Rd
\end{equation}
with the magnetic vector potential $b(x)$ and $n(x)$ a variable index of refraction that does not necessarily converge to a constant at infinity, but can have an angular dependency like $n(x) \to n_{\infty}\left(\frac{x}{|x|}\right)$ as $|x|\to\infty$. We prove an explicit Sommerfeld radiation condition
\begin{equation}\notag
\int_{\Rd} \left|\D u - in_{\infty}^{1/2}\frac{x}{|x|}u\right|^{2} \frac{dx}{1+|x)} < + \infty 
\end{equation}
for solutions obtained from the limiting absorption principle and we also give a new energy estimate
\begin{equation}\notag
\int_{\Rd}\left| \nabla_{\omega}n_{\infty}\left(\frac{x}{|x|}\right)\right|^{2}\frac{|u|^{2}}{1+|x|} dx < +\infty, 
\end{equation}
which explains the main physical effect of the angular dependence of $n$ at infinity and deduces that the energy concentrates in the directions given by the critical points of the potential. 
\end{abstract}

\date{\today}


\maketitle

\section{Introduction}

Let us consider the following Helmholtz equation
\begin{equation}\label{***}
(\nabla + ib(x))^{2}u(x) + n(x)u(x) = f(x), \quad x\in \Rd
\end{equation}
where $b=(b_{1},\ldots,b_{d}):\Rd \to \Rd$ is a magnetic potential and $n: \Rd \to \R$ is a variable index of refraction that admits a radial limit
\begin{equation}\label{radiallimit}
n(x) \to n_{\infty}\left(\frac{x}{|x|} \right) \quad \textrm{as} \quad |x|\to \infty.
\end{equation}

In this work we are interested in the study of the existence and uniqueness of solution of the equation (\ref{***}) with an appropriate radiation condition using the limiting absorption method as well as some new estimates that characterize the behavior of the solution at infinity. 

The radiation conditions are known to be necessary for the uniqueness of solutions to (\ref{***}), as first remarked by Sommerfeld \cite{So} in the free case $b \equiv 0 \equiv V$. Moreover, there is a strong connection between this topic and the limiting absorption principle for Sch\"odinger operators, to which a lot of research has been devoted (see for example, \cite{E1}, \cite{IS}, \cite{Ku}, \cite{A}, \cite{S1}, \cite{M1}, \cite{M2}, \cite{MU}, \cite{S2}, \cite{I}, \cite{Be}, \cite{E2}, \cite{S}, \cite{PV2}, \cite{RT}, \cite{Z}). Let us now give a brief picture about the recent advances on Sommerfeld radiation.

As a first result we mention Eidus \cite{E1}, where it is showed that there exists a unique solution $u(\lambda, f)$ of the equation (\ref{***}) with $n(x) = \lambda + V(x)$ in $\R^{3}$ satisfying the radiation condition
\begin{equation}
\lim_{r\to \infty} \int_{|x|=r} \left|\frac{\partial u}{\partial|x|}-i\lambda^{1/2}u \right|^{2} d\sigma(r)=0.
\end{equation}
Here $b_{j}(x)$ is assumed to vanish close to infinity and the electric potential satisfies $V(x)=O(|x|^{-2-\alpha})$ with $\alpha > \frac{1}{6}$ at infinity. In 1972, Ikebe and Saito \cite{IS} extended the result by Eidus to electric potentials of the form $V= \lambda\tilde{p} + Q$, where $\tilde{p}$ is long range and $Q$ is short range, obtaining the precise radiation condition
\begin{equation}\label{is1}
\int_{\Rd} \left|(\nabla + ib)u - i\lambda^{1/2}\frac{x}{|x|}u \right|^{2}\frac{dx}{(1+|x|)^{1-\delta}} < +\infty,
\end{equation}
where $0<\delta<1$ is a fixed constant. This result has been recently improved by Zubeldia in \cite{Z}, by adding some singularities on the potentials at the origin and extending the range of $\delta$ to $0<\delta<2$, in (\ref{is1}). In the purely electric case $b\equiv 0$, under similar assumptions on $V$, Saito proved in \cite{S1}, \cite{S2} another type of radiation condition,
\begin{equation}\notag
\int_{\Rd} |\nabla u - i(\nabla K)u|^{2} \frac{dx}{(1+|x|)^{1-\delta}} < + \infty,
\end{equation}
being $K=K(x,\lambda)$ an exact or approximate solution of the eikonal equation
\begin{equation}\label{eikointro}
|\nabla K|^{2} = \lambda(1+\p(x)).
\end{equation}
Here the principal order of $\nabla K$ has the form
\begin{equation}\notag
\nabla K(x,\lambda)=\Phi(x,\lambda)\frac{x}{|x|}.
\end{equation}
Notice that the unit normal $\frac{x}{|x|}$ to the sphere appears in the radiation condition, in all the above mentioned papers.

In 1987, Saito considers \cite{S} more general long range potentials and proves the existence of a unique solution of the equation (\ref{***}) for $n(x)= \lambda(1+ \p(x)) + Q(x)$ with a 
nonspherical radiation condition
\begin{equation}\label{Saito}
\int_{|x|\geq 1} \left|(\nabla + ib)u - i\sqrt{\lambda}\nabla K u\right|^{2}\frac{dx}{(1+|x|)^{1-\delta}} < +\infty
\end{equation}
where $\nabla K$ is the outward normal of a surface which is not a sphere in general and satisfies the eikonal equation (\ref{eikointro}). More precisely, he considers potentials $p(x)=\lambda\p(x)$ such that $p(x)=O(1)$, $\frac{\partial p}{\partial x_{j}}=O(|x|^{-1})$ and $\frac{\partial^{2}p}{\partial x_{i} \partial x_{j}} = O(|x|^{-2})$ at infinity. Under the same assumptions on the potentials, Barles establishes in \cite{B} the existence of solution of the corresponding eikonal equation for $|x| > R_{0}$, for $R_{0}, \lambda$ large enough.

Some years later, in 2007, Perthame and Vega \cite{PV2} showed that, in the purely electric case $b \equiv 0$, if one puts some further restrictions on the potential $n(x)$ as
\begin{align}\notag
2\sum_{j\in \mathbb{Z}} \sup_{C(j)} \frac{(x\cdot \nabla n(x))_{-}}{n(x)} < 1,
\end{align}
where $C(j)$ denotes the annulus $\{2^{j-1}\leq |x| \leq 2^{j}\}$, while $(a)_{-} = -\min \{0, a\}$ is the negative part of $a\in\R$ and if there exists $n_{\infty} \in C^{3}(S^{d-1})$ such that $n_{\infty}(\omega) \geq n_{0} >0$ with
\begin{equation}\label{barrixalim}
\left| n(x) - n_{\infty}(\omega) \right| \leq n_{\infty}(\omega)\frac{\Gamma}{|x|}, \quad \Gamma >0, \quad n>0,
\end{equation}
then the following precise radiation condition holds
\begin{equation}\notag
\int_{|x|\geq 1} \left|\nabla u - in_{\infty}^{1/2}\frac{x}{|x|}u\right|^{2} \frac{dx}{|x|} < +\infty.
\end{equation}
Note that the spherical term
\begin{equation}
n_{\infty}^{1/2}\left(\frac{x}{|x|}\right)\frac{x}{|x|}
\end{equation}
appears in this formula instead of the gradient of the phase as in (\ref{Saito}). This apparent contradiction can be explained by the existence of some extra energy estimate already announced in \cite{PV4}) of the form
\begin{equation}\label{energyintro}
\int_{\Rd} \left| \nabla_{\omega}n_{\infty}\left( \frac{x}{|x|}\right) \right|^{2}\frac{|u|^{2}}{1+|x|} dx< +\infty,
\end{equation}
for some $\omega = \frac{x}{|x|}\in \mathbb{S}^{d-1}$. Here we define
\begin{equation}\notag
\nabla_{\omega}n(\omega) = \frac{\partial}{\partial\omega}n(\omega) : = |x|\nabla^{\bot}n\left(\frac{x}{|x|}\right)
\end{equation}
and
\begin{equation}\notag
\nabla^{\bot} u(x) = \nabla u(x) - \frac{x}{|x|}\nabla^{r}u(x), \quad \nabla^{r} u(x)=\partial_{r}u(x) := \frac{x}{|x|}\cdot\nabla u(x). 
\end{equation}
Estimate (\ref{energyintro}) says that the points where $|\nabla_{\omega}n_{\infty}(\omega)|$ vanishes on the sphere are the concentration directions for the energy $|u|^{2}$. In other words, the energy is not dispersed in all directions but concentrated on those given by the critical points of the potential. Thus we see that in this case the behavior of the solution at infinity can be very different to the one exhibited by free solutions.

The role played by the critical points of $n_{\infty}$ was already pointed out by Herbst \cite{He}, where is considered the case when $n(x) = \lambda + V(x)$ with
\begin{equation}\notag
V(x)=|x|^{-\sigma}V\left(\frac{x}{|x|}\right) \quad 0<\sigma < 2, \quad \forall x\in \Rd\backslash \{0\}.
\end{equation}
This potential is also studied in \cite{FG}, \cite{GVV} and \cite{HS}, for the study of the counterexamples of Strichartz inequalities for Schr\"odinger equations with repulsive potentials and the existence and completeness of the wave operator.

Let us introduce some notations. For $f:\Rd \to \mathbb{C}$, we define the norms
\begin{equation}\notag
|||f|||_{R_{0}} := \sup_{R>R_{0}} \left(\frac{1}{R}\int_{|x|\leq R} |f(x)|^{2}\right)^{1/2}
\end{equation}
and 
\begin{equation}\notag
N_{R_{0}}(f) := \sum_{j > J}\left(2^{j+1}\int_{C(j)} |f|^{2} \right)^{1/2} + \left(R_{0}\int_{|x|\leq R_{0}} |f|^{2} \right)^{1/2}
\end{equation}
where $J$ is such that $2^{J-1} < R_{0} < 2^{J}$ and $C(j)=\{x\in\Rd : 2^{j}\leq |x|\leq 2^{j+1}\}$. The norms $|||f|||_{1}$ and $N_{1}(f)$ are known as Agmon-H\"ormander norms. We drop the index $R_{0}$ if $R_{0}=0$, getting then the Morrey-Campanato norm and its dual in the sense that 
\begin{equation}\notag
\left| \int fg \right|  \leq |||f|||N(g).
\end{equation}
Moreover, we denote the magnetic gradient by
$$
\nabla_{b} := \nabla + i b
$$
and its tangential component by
$$
|\D^{\bot} u|^{2} = |\D u|^{2} - |\D^{r}u|^{2}, \quad \quad \D^{r}u = \frac{x}{|x|}\cdot \D u.
$$
We also introduce the magnetic field $B(x)$ associated to the magnetic potential $b(x)$ which is given by the $d\times d$ anti-symmetric matrix defined by 
\begin{equation}\notag
B=(Db)-(Db)^{t}, \quad B_{jk}=\left(\frac{\partial b_{j}}{\partial x_{k}}- \frac{\partial b_{k}}{\partial x_{j}}\right) \quad j, k=1,\ldots,d.
\end{equation}
The tangential part of the magnetic field $B$ is given by
$$
B_{\tau}:= \frac{x}{|x|}B(x), \quad \quad (B_{\tau})_{j}=\sum_{k=1}^{d} \frac{x_{k}}{|x|} B_{kj}.
$$
This quantity was introduced by Fanelli and Vega \cite{FV} and is related to singular magnetic potentials. See \cite{FV} and \cite{Z1} for more details.

One of the main contributions of this work is to extend the new energy estimate to the magnetic case, inspired to Perthame and Vega \cite{PV2}. Let us consider the magnetic Helmholtz equation
\begin{equation}\label{*eps}
(\nabla + ib(x))^{2}u(x) + n(x)u(x) + i\varepsilon u(x) = f(x), \quad \varepsilon >0.
\end{equation}
Then our goal is to show the estimate (\ref{energyintro}) for the solution $u$ of this equation using integration by parts.

To this end, we first need to prove suitable a-priori estimates to the solution of the equation (\ref{*eps}). On the one hand, one needs to control the Morrey-Campanato norm of the solution and its magnetic gradient. On the other hand, we emphasize that the estimate for the tangential component of the magnetic gradient
$$
\int_{\Rd} \frac{|\D^{\bot}u|^{2}}{|x|} dx < \infty
$$
turns out to be fundamental. For this purpose, we will require that $n(x)$ and the tangential component of the magnetic field satisfy the condition
\begin{equation}\label{nbtau}
2\sum_{j\in \mathbb{Z}} \sup_{C(j)} \frac{(x\cdot \nabla n(x))_{-} + 2^{2j}|B_{\tau}|^{2}}{n(x)} < 1.
\end{equation}
However, for the energy estimate it will be necessary to put some further restrictions to $n(x)$ as in \cite{PV2} and also to the magnetic field $B$. In fact, we assume that
\begin{equation}\label{assubjk}
\sum_{j\geq 0} 2^{2j}|B_{jk}|^{2} < \infty.
\end{equation}
Moreover, it is required that
\begin{align}\label{1.12}
\textrm{there exists} \quad n_{\infty}\left(\frac{x}{|x|}\right) \in C^{\infty}\left(S^{d-1} \right), \quad n_{\infty}\left(\frac{x}{|x|} \right) \geq n_{0} >0,
\end{align}
and
\begin{align}
\left|n(x) - n_{\infty}\left(\frac{x}{|x|}\right) \right| \leq n_{\infty}\left(\frac{x}{|x|} \right)\frac{\Gamma}{|x|}, \quad \quad \Gamma >0, \quad n>0.\label{1.13'}
\end{align}
Note that from (\ref{1.12}) and (\ref{1.13'}) it may be concluded that 
\begin{equation}\label{condin}
|n| \leq C \quad \quad \textrm{and} \quad \quad n\geq \frac{n_{0}}{2} \quad \textrm{for} \quad |x| \quad \textrm{large enough}.
\end{equation}

We may now state the first result of this paper. 

\begin{thm}\label{Theoremenergy1}
For dimensions $d\geq 3$, we assume (\ref{nbtau})-(\ref{1.13'}). Then the solution of the Helmholtz equation (\ref{*eps}) satisfies, for $R\geq 1$ large enough
\begin{equation}\label{enerthm}
\int_{|x|\geq R} \left|\nabla_{\omega}n_{\infty}\left(\frac{x}{|x|}\right) \right|^{2}\frac{|u|^{2}}{|x|} \leq C(1+\varepsilon)\left(N\left(\frac{f}{n^{1/2}}\right)\right)^{2},
\end{equation}
for some constant $C$ independent of $\varepsilon$ and $n$.
\end{thm}

This theorem is the natural generalization of the result by Perthame and Vega in \cite{PV2}, which as far as we know was not known for first order perturbations of the Helmholtz equation. Nevertheless, it does not seem that these conditions on the potentials are sufficient to prove the limiting absorption principle for the equation (\ref{***}).

In order to get the limiting absorption principle for the electromagnetic Helmholtz equation with long range potentials that in particular include those which are homogeneous of degree zero, we will follow Saito \cite{S}. Let us consider the equation
\begin{equation}\label{**s}
(\nabla + ib(x))^{2} u + n(x)u + Q(x) u = f,
\end{equation}
with $n(x)= \lambda (1+ \p(x))$ where $\p, Q:\Rd \to \R$ can be interpreted as electric potentials. Under suitable assumptions on the potentials, we will prove the existence of a unique solution of the equation (\ref{**s}) satisfying a specific Sommerfeld radiation condition together with some a-priori estimates of Agmon-H\"ormander type. We will use multiplier techniques based on integration by parts, inspired by \cite{PV1}, \cite{F} and \cite{Z}. We work with potentials that decay as in \cite{S} at infinity and the most important issue is that we allow singularities on the potentials at the origin. 

It is worth pointing out that the self-adjointness of the electromagnetic hamiltonian
$$
L = (\nabla + ib(x))^{2} + V(x)
$$
with $V(x)= \lambda\p(x) + Q(x)$ is necessary for the limiting absorption principle. For this purpose, we need to require some local integrability conditions on our potentials. In what follows we will always assume that
\begin{equation}
b_{j} \in L^{2}_{loc}, \quad V\in L^{1}_{loc}, \quad \quad \int V|u|^{2} \leq \nu \int |\nabla u|^{2}, \quad 0<\nu<1.
\end{equation}
As a consequence, it follows (see \cite{Z1}, chapter 1 for more details) that $L$ is self-adjoint in $L^{2}(\Rd)$ with form domain
$$
D(L)= \{f\in L^{2}(\Rd) : \int |\D f|^{2} - \int V|f|^{2} < \infty\}.
$$

We can now state the second main result of the paper.

\begin{thm}\label{d>3}
Let $d\geq 3$, $\p \in C^{2}(\Rd \backslash \{0\})$, $r_{0}\geq 1$ and $\mu >0$. We assume that
\begin{equation}\label{condicionf} 
|B_{jk}(x)| + |Q(x)| \leq \frac{c}{|x|^{1+\mu}}, \quad \textrm{if} \quad |x|\geq r_{0},
\end{equation}
\begin{equation}\label{(Q)}
|Q(x)|\leq \frac{c}{|x|^{2-\alpha}} \quad \textrm{if} \quad |x|\leq r_{0}, \quad 0< \alpha < 2,
\end{equation}
\begin{equation}\label{assuv1}
|\partial^{\alpha}\p(x)| \leq C^{*}|x|^{-\alpha} \quad (|\alpha|\leq 2),
\end{equation}
for some $c>0$, where $\alpha=(\alpha_{1},\ldots,\alpha_{d})$ is an arbitrary multi-index with nonnegative integers $\alpha_{j}$ ($1\leq j \leq d$), $|\alpha|=\alpha_{1} +  \cdots + \alpha_{d}$, $\partial^{\alpha}=\partial_{1}^{\alpha_{1}}\cdots\partial_{d}^{\alpha_{d}}$ and $C^{*}$ is a positive small constant ($0<C^{*} < 1$). In addition, let $c_{1}$ small enough and we consider
\begin{equation}\label{(b1)}
|B| \leq \frac{c}{|x|^{2-\alpha}} \quad  \quad |x|\leq r_{0}, \quad 0<\alpha<2,
\end{equation}
if $d=3$ and
\begin{equation}\label{(b)}
|B| \leq \frac{c_{1}}{|x|^{2}} \quad  \quad |x|\leq r_{0},
\end{equation}
if $d>3$. We also require that the magnetic potential satisfies the condition
\begin{equation}\label{gaugesaito}
|\nabla\cdot A| \leq c|x|^{-2},
\end{equation}
for some $c>0$. Then, for any $\lambda \in [\lambda_{0}, \lambda_{1}]$ with $0<\lambda_{0} < \lambda_{1} < \infty$, there exists a unique solution of the Helmholtz equation (\ref{**s}) satisfying
\begin{align}\label{i}
&\lambda|||u|||_{1}^{2} + |||\D u|||_{1}^{2} \leq C (N_{1}(f))^{2}
\end{align}
and the radiation condition
\begin{equation}\label{radiacion}
\int_{|x|\geq 1}  \left|\D u- i\lambda^{1/2}\nabla K u\right|^{2} \frac{dx}{(1+|x|)^{1-\delta}} \leq C \int_{|x|\geq 1/2}(1+|x|)^{1+\delta}|f|^{2}dx,
\end{equation}
for any $0<\delta\leq 1$ such that $\delta<\mu$, where $C=C(\lambda_{0})$ and $K$ is solution of the eikonal equation
\begin{equation}\label{eiko2}
|\nabla K|^{2}=1 + \p(x).
\end{equation}
\end{thm}

\begin{remark}
We point out that condition (\ref{gaugesaito}) is only needed for the unique continuation property, which is fundamental for proving the uniqueness result (see Theorem \ref{unicidad2} below). We use a unique continuation result proved by Regbaoui \cite{R}.
\end{remark}

Theorem \ref{d>3} is the analog to the result by Saito in \cite{S}, generalized to possibly singular potentials. Observe that in our approach it will be necessary to solve the eikonal equation (\ref{eiko2}). 

Barles \cite{B} proved that under the assumption (\ref{assuv1}) and for $C^{*}$ small enough, there exists a solution of the equation (\ref{eiko2}) for $|x|>R_{0}$ with $R_{0}$ large enough, see section \ref{eikonalequation}. In general, one can not expect that the vector $\nabla K$ points at the direction $\frac{x}{|x|}$. An illustrative example given by Saito \cite{S} is to consider
\begin{equation}\notag
\p(x)=-\frac{1}{\lambda}\frac{x_{1}}{|x|}.
\end{equation}
Then $\p(x)$ satisfies (\ref{assuv1}) for $\lambda$ large enough and
\begin{equation}\notag
K(x) = a(\lambda)|x| - b(\lambda)x_{1}
\end{equation}
with
\begin{equation}\notag
\left\{ \begin{array}{ll} 
a(\lambda)=\frac{1}{2}[(1+1/\lambda)^{1/2} + (1-1/\lambda)^{1/2}] \\ 
b(\lambda)=\frac{1}{2}[(1+1/\lambda)^{1/2} - (1-1/\lambda)^{1/2}] 
\end{array} \right.
\end{equation} 
is a solution of the eikonal equation (\ref{eiko2}). This boundary condition differs from ours in all points except when $\nabla n =0$ ($n=\lambda + \p$). In the trivial case $\p(x)=0$, one can take $K(x,\lambda) = |x|$.

Note that assumptions needed to obtain the energy estimate and those for the limiting absorption principle are different and not comparable. On the one hand, if $n=n_{\infty}$ and regular, (\ref{nbtau}) is trivially fulfilled. On the other hand, condition (\ref{assuv1}) with $n=\lambda(1+\p)$ does not imply the existence of the limit $n_{\infty}$. In addition, condition (\ref{barrixalim}) does not need any regularity assumption on $n(x)$ as in (\ref{assuv1}). It is easy to see that if besides (\ref{assuv1}), we require 
\begin{equation}\label{p1}
|\partial_{r}\p(x)| \leq c_{2}|x|^{-1-\mu}, \quad \quad \quad |x|\geq 1,
\end{equation}
for some $c_{2}>0$, $\mu >0$, then the index of refraction $n(x)$ admits a radial limit $n_{\infty}\left( \frac{x}{|x|}\right)$ as $|x|\to\infty$. Moreover, it follows that
\begin{equation}\notag
\left|n(r\omega) - n_{\infty}(\omega) \right| \leq \Gamma|x|^{-\mu}
\end{equation}
for $\Gamma >0$, where $r=|x|$ and $\omega=\frac{x}{|x|}$. See \cite{PV2} for more details.

A combination of Theorem \ref{Theoremenergy1} with Theorem \ref{d>3} will allow us to deduce an explicit Sommerfeld radiation condition. Let us consider the Helmholtz equation  
\begin{equation}\label{****}
(\nabla + ib)^{2} u + \lambda(1+ \p)u = f
\end{equation}
and we denote $n(x) = \lambda(1+ \p(x))$. The following result complements that of Saito \cite{S} when $Q=0$ and extends the one given in \cite{PV2} to the magnetic case. 

\begin{thm}\label{teoremaexplicit}
For dimensions $d\geq 3$, assume (\ref{nbtau}) and (\ref{assuv1}). Then for sufficiently small $C^{*}>0$ and for any $\lambda \in [\lambda_{0}, \lambda_{1}]$ with $0<\lambda_{0} < \lambda_{1}<\infty$, there exists a unique solution of the Helmholtz equation (\ref{****}) satisfying
\begin{equation}\label{nradi}
\int_{\Rd} \left|\D u - i n^{1/2}(x)\frac{x}{|x|}u \right|^{2}\frac{dx}{|x|} \leq C_{\delta}\int_{\Rd} (1+|x|)^{1+\delta} |f|^{2}dx,
\end{equation}
for some $\delta >0$. Moreover, if there exist $n_{\infty}$, $\Gamma>0$ and $\mu>0$ such that
\begin{equation}\notag
\left|n(x) - n_{\infty}\left(\frac{x}{|x|} \right)\right| \leq n(x)\frac{\Gamma}{|x|^{\mu}}\quad  \textrm{for} \quad |x| \quad \textrm{large enough},
\end{equation}
then it follows that
\begin{align}\label{explicit}
\int_{|x|\geq 1} \left|\D u - i n_{\infty}^{1/2}\frac{x}{|x|}u\right|^{2} \frac{dx}{|x|} \leq C\int_{\Rd} (1+|x|)^{1+\delta}|f|^{2}dx.
\end{align}
\end{thm}

Note that the spherical term
\begin{equation}\notag
n_{\infty}^{1/2}(\omega)\frac{x}{|x|}
\end{equation}
appears in this formula instead of the phase as in (\ref{radiacion}), where $\nabla K$ is the outward normal of the surface $|K(x, \lambda)|=\lambda$, which is not necessarily a sphere. This apparent contradiction can be explained by the extra estimate (\ref{enerthm}) on the energy decay which applies for the above example. In fact, it can be interpreted as a concentration of the energy along the directions given by the critical points of $n_{\infty}$. In other words, the Sommerfeld condition hides the main physical effect arising for a variable $n$ at infinity; the energy concentration on lines rather than dispersion in all directions.

The rest of the paper is organized as follows. In the next section, we give a brief exposition of the eikonal equation and its properties that will be useful for the proofs of the main results. Section \ref{newenergy} will be concerned with the new energy estimate. We will prove Theorem \ref{Theoremenergy1} showing first the appropriate a-priori estimates given in Theorem \ref{Theorem1.1} that permits to deduce the desired conclusion. In section \ref{LAPenergy} we proceed with the study of the limiting absorption principle for the equation (\ref{**s}). We conclude the proof of Theorem \ref{d>3}, following the ideas in \cite{Z}. Section \ref{sectionexplicit} provides a detailed proof of the new explicit Sommerfeld condition given in Theorem \ref{teoremaexplicit}. The fundamental tool of the proofs are Morawetz-type abstract identities based on integration by parts, which are established in Appendix (see Lemmas \ref{appendix1} and \ref{appendix2}).

\subsection*{Notation}
Throughout the paper, $C$ denotes an arbitrary positive constant and $\kappa$ stands for a small positive constant. In most of the cases, $\kappa$ will come from the inequality $ab \leq \kappa a^{2}+ \frac{1}{4\kappa}b^{2}$, which is true for arbitrary $\kappa >0$. In the integrals where we do not specify the integration space we mean that we are integrating in the whole $\Rd$ with respect to the Lebesgue measure $dx$, i.e. $\int = \int_{\Rd} dx$.

\section{The Eikonal Equation}\label{eikonalequation}

In order to determine the phase arising in the Sommerfeld radiation condition (\ref{radiacion}) and to conclude the explicit one (\ref{explicit}), we need to solve the eikonal equation
\begin{equation}\label{eikonal}
|\nabla K|^{2} = 1+\p(x), \quad x\in\Rd.
\end{equation}

Setting 
\begin{equation}\label{g}
g(x,C^{*})=|x|^{-1}K(x,C^{*}),
\end{equation}
we derive the following Hamilton-Jacobi equation 
\begin{equation}\label{HJ}
|g|^{2} + 2r\partial_{r}g + |x|^{2}|\nabla g|^{2} = 1+\p(x), \quad x\in \Rd\backslash\{0\}.
\end{equation}
From (\ref{HJ}), under assumption $(\ref{assuv1})$, Barles showed in \cite{B} that there exists $C_{0} > 0$ such that for any $C^{*} <  C_{0}$ the differential equation (\ref{eikonal}) has a solution $\varphi=\varphi(x,C^{*})$ for $|x|\geq r_{0}$ which satisfies
\begin{itemize}

\item[(i)] $\varphi(x,C^{*})$ is a real-valued $C^{3}$ function for $|x| \geq r_{0}$.

\item[(ii)] For any $|x|\geq r_{0}$ and $C^{*}<C_{0}$,
\begin{equation}\label{g}
c_{0} \leq g(x,C^{*}) \leq c_{1}
\end{equation}
\noindent
with positive constants $c_{0}$ and $c_{1}$.

\item[(iii)] When $C^{*} \to 0$,

\begin{equation}\label{propietateg}
|x|^{j}(\partial^{j}g)(x,C^{*}) \longrightarrow \left\{ \begin{array}{ll}
1, & \quad j=0\\
0, & \quad j=1,2,3
\end{array}\right.
\end{equation}
uniformly for $x \in \{x\in\Rd : |x| \geq r_{0}\}$. 

\end{itemize}

Therefore, one can easily deduce the following identity that will be very useful in section \ref{LAPenergy}.

\begin{lem}\label{lema2.5}(\cite{S}, Lemma 2.5)
For the solution $K$ of the eikonal equation (\ref{eikonal}) and for $1\leq i,j\leq d$, the following identity holds 
\begin{equation}\label{segunda}
\frac{\partial^{2} K}{\partial x_{i}\partial x_{j}} = \frac{|\nabla K|^{2}}{K}\delta_{ij} -\frac{1}{K}\frac{\partial K}{\partial x_{i}}\frac{\partial K}{\partial x_{j}} +\frac{1}{K}F_{ij}(x,C^{*}),
\end{equation}
where $F_{ij}(x,C^{*})$ is a bounded function of $x$ for $|x|\geq r_{0}$ such that
\begin{equation}\label{0.2}
\lim_{C^{*} \to 0} \sup_{|x|\geq r_{0}} |F_{ij}(x,C^{*})| = 0 \quad (i,j=1,\ldots,d)
\end{equation}
\noindent
and
\begin{displaymath}
\delta_{ij} = \left\{ \begin{array}{ll}
1 & \quad i=j,\\
0 & \quad i \not= j.
\end{array} \right.
\end{displaymath}
\end{lem}

\begin{proof}
Setting $K(x,C^{*})=|x|g(x,C^{*})$ and $\tilde{x}_{k}=\frac{x_{k}}{|x|}$, we have
\begin{equation}
\label{2.35}
\frac{\partial K}{\partial x_{i}} = \tilde{x}_{i}g + |x|\frac{\partial g}{\partial x_{i}}.
\end{equation}
Then,
\begin{equation}\notag
\frac{\partial^{2} K}{\partial x_{i}\partial x_{j}} = \frac{\delta_{ij}g}{|x|} -\tilde{x}_{i}\tilde{x}_{j}\frac{g}{|x|} + \tilde{x}_{i}\frac{\partial g}{\partial x_{j}} +\tilde{x}_{j}\frac{\partial g}{\partial x_{i}} +|x|\frac{\partial^{2}g}{\partial x_{i}\partial x_{j}},
\end{equation}
which can be written as
\begin{equation}\label{0.4}
\frac{\partial^{2}K}{\partial x_{i}\partial x_{j}} = \frac{\delta_{ij}}{K}g^{2} - \frac{\tilde{x}_{i}\tilde{x}_{j}}{K}g^{2} + \frac{1}{K}G_{ij}(x,C^{*}),
\end{equation}
with
\begin{equation}\notag
G_{ij}(x,C^{*}) =\left( \tilde{x}_{i}|x|\frac{\partial g}{\partial x_{j}}+ \tilde{x}_{j}|x|\frac{\partial g}{\partial x_{i}} + |x|^{2}\frac{\partial^{2}g}{\partial x_{i}\partial x_{j}}\right).
\end{equation}
On the other hand, from (\ref{2.35}) it follows that
\begin{displaymath}
\left\{ \begin{array}{ll}
\tilde{x}_{i}g=\frac{\partial K}{\partial x_{i}} - |x|\frac{\partial g}{\partial x_{i}},\vspace{0.2cm}\\
|\nabla K|^{2}=g^{2} +2g|x|\tilde{x}\cdot\nabla g +|x|^{2}|\nabla g|^{2}.
\end{array} \right.
\end{displaymath}
Thus we obtain
\begin{equation*}
\frac{\tilde{x}_{i} \tilde{x}_{j}}{K} g^{2}= \frac{1}{K} \frac{\partial K}{\partial x_{i}}\frac{\partial K}{\partial x_{j}} -
\frac{1}{K} \left( |x|\frac{\partial g}{\partial x_{i}}\frac{\partial K}{\partial x_{j}} 
+ |x|\frac{\partial K}{\partial x_{i}}\frac{\partial g}{\partial x_{j}} - |x|^{2}\frac{\partial g}{\partial x_{i}}\frac{\partial g}{\partial x_{j}}\right),\vspace{0.2cm}
\end{equation*}
which together with (\ref{0.4}), gives (\ref{segunda}) with
\begin{align}\label{F}
F_{ij}&= G_{ij}-\delta_{ij}(2|x| \tilde{x}\cdot\nabla g) + |x|^{2}|\nabla g|^{2}) + |x|\frac{\partial g}{\partial x_{i}}\frac{\partial K}{\partial x_{j}}\\
& +|x|\frac{\partial K}{\partial x_{i}}\frac{\partial g}{\partial x_{j}} - |x|^{2}\frac{\partial g}{\partial x_{i}}\frac{\partial g}{\partial x_{j}}.\notag
\end{align}
The relation (\ref{0.2}) follows from (\ref{propietateg}) and the lemma is proved.

\end{proof}

In order to prove the explicit condition (\ref{explicit}) we shall deduce (see section \ref{sectionexplicit}) the estimate
\begin{equation}\label{tangeneikonal}
\int |\nabla^{\bot} K u|^{2}\frac{1}{1+|x|} < +\infty.
\end{equation}
This is an energy estimate in itself which says that $u$ concentrates along the critical points of $\nabla{\bot}K$. In fact, from the hypotheses (\ref{p1}) for the potential $\p(x)$, it follows that these critical point coincide with those of $\nabla_{\omega}n_{\infty}$ establishing a relation between the energy estimate (\ref{enerthm}) and (\ref{tangeneikonal}). 

\begin{lem}(\cite{PV2}, Theorem 3.2) Under assumptions (\ref{assuv1}) and (\ref{p1}), the solution to (\ref{HJ}) satisfies for $C^{*}$ small enough and $x\neq 0$
\begin{equation}
|\partial_{r}g| \leq C^{*}r^{-1-\mu},
\end{equation}
and $g\left(r\frac{x}{|x|}\right) \to g_{\infty}\left(\frac{x}{|x|}\right)$ as $r\to \infty$, a smooth solution to the equation
\begin{equation}
g_{\infty}(\omega)^{2} + |\nabla_{\omega}g_{\infty}(\omega)|^{2} = n_{\infty}(\omega), \quad \omega\in S^{d-1}.
\end{equation}
Moreover,
\begin{equation}
|\nabla^{\bot}K|=|\nabla_{\omega}g_{\infty}(\omega)| + O(r^{-\mu})
\end{equation}
and
\begin{equation}
0< C_{1}|\nabla_{\omega}g_{\infty}| \leq |\nabla_{\omega}n_{\infty}| \leq C_{2}|\nabla_{\omega}g_{\infty}|.
\end{equation}
\end{lem}

Observe however that for the energy estimate (\ref{enerthm}) we do not need the existence of a solution to the eikonal equation (\ref{eikonal}) which could well not exist.

\section{The energy estimate. Proof of Theorem \ref{Theoremenergy1}}\label{newenergy}

The purpose of this section is to extend the result by Perthame and Vega \cite{PV2} to the magnetic case. To be more precise, we are interested in proving the energy estimate (\ref{enerthm}) given in Theorem \ref{Theoremenergy1} for solutions $u\in H^{1}_{A}(\Rd)$ of the magnetic Helmholtz equation
\begin{equation}\label{**hauda}
\D^{2} u +  n(x)u + i\varepsilon u = f(x), \quad \quad \varepsilon >0.
\end{equation}
This estimate uses in a strong way the a-priori estimate for the Morrey-Campanato norm of the solution $u$ of the equation (\ref{**hauda}) as well as the estimate for the tangential part of its magnetic gradient.

\vspace{0.2cm}
Let us consider $n(x) >0$ such that
\begin{align}
& n= n_{1} + n_{2} \quad \quad \textrm{with} \quad \quad n_{2}\in L^{\infty},\label{1.7}\\
& \Vert n_{1}^{1/2}u\Vert_{L^{2}} \leq (1-c_{0})\Vert \nabla u \Vert_{L^{2}} \quad \textrm{for some $c_{0}>0$},\label{1.8}\\
& 2\sum_{j\in\mathbb{Z}} \sup_{C(j)} \frac{(x\cdot \nabla n(x))_{-} + 2^{2j}|B_{\tau}|^{2}}{n(x)} :=\beta < 1, \label{1.9}
\end{align}
where $C(j) = \{x\in \Rd : 2^{j-1}\leq |x| \leq 2^{j}\}$ and $(a)_{-}$ denotes the negative part of $a\in \mathbb{R}$.

Then it follows the following result.

\begin{thm}\label{Theorem1.1}
Let $d\geq 3$ and assume that (\ref{1.7})-(\ref{1.9}) hold. Then the solution to the Helmholtz equation (\ref{**hauda}) satisfies
\begin{align}\label{1.10}
M^{2} := |||\D u|||^{2} &+ ||| n^{1/2}u|||^{2} + \int \frac{|\D^{\bot}u|^{2}}{|x|}\\
& \leq C(\varepsilon + \Vert n_{2}\Vert_{L^{\infty}}) \left(N\left(\frac{f}{n^{1/2}} \right) \right)^{2},\notag
\end{align}
where $C$ is independet of $\varepsilon$ and $n$.
\end{thm}

\begin{proof}
The proof is based on the identities which are stablished in Appendix and follows the same arguments of the proofs of Theorem 2.1 in \cite{Z} or Theorem 1.1 in \cite{PV2}. Thus we give only the main ideas of the proof.

Let $R>0$ and we consider the functions $\psi$ and $\varphi$ given by
\begin{equation}\notag
\nabla\psi(x)  = \left\{ \begin{array}{ll}
\frac{|x|}{R} & \textrm{if $|x| \leq R$},\vspace{0.1cm} \\
\frac{x}{|x|} & \textrm{if $|x| \geq R$},
\end{array} \right.
\end{equation}
\begin{equation}\notag
\varphi(x) = \left\{ \begin{array}{ll}
\frac{1}{2R} & \textrm{if $|x| \leq R$},\\
0 & \textrm{if $|x| \geq R$}.
\end{array} \right.
\end{equation}

Let us add the identity (\ref{(4.3)}) to (\ref{(4.11)}) with the above choices of the multipliers, respectively. Then, analysis similar to that in the proof of Theorem 2.1 in \cite{Z} gives
\begin{align}\label{idergy}
\frac{1}{2R}&\int_{|x|\leq R} (|\D u|^{2} + n(x)|u|^{2}) + \int_{|x|\geq R} \frac{|\D^{\bot}u|^{2}}{|x|} + \frac{(d-1)}{8R^{2}}\int_{|x|=R} |u|^{2}\\
& \leq \frac{1}{2}\int (\partial_{r}n)_{-} |u|^{2} + \frac{1}{R}\int_{|x|\leq R} |x||B_{\tau}||\D u||u| +\int_{|x|\geq R} |B_{\tau}||\D^{\bot}u||u| \notag\\
& + \varepsilon \int |\D u||u|+ 2\int |f||\D u| + C\int \frac{|f||u|}{|x|}.\notag
\end{align}

The terms related to $f$ can be treated as in the proof of Theorem 2.1 in \cite{Z}, obtaining
\begin{align}\label{fterm2}
\int \frac{|f||u|}{|x|} + \int |f||\D u| &\leq \kappa \left( |||\D u|||^{2} + \sup_{R>0}\frac{1}{R^{2}}\int_{|x|=R} |u|^{2}\right)\\
& + C_{\kappa}(N(f))^{2},\notag
\end{align}
where $\kappa$ denotes an arbitrary positive small constant.

Let us study the potential terms. On the one hand, we have
\begin{align}
\frac{1}{2}\int (\partial_{r}n)_{-}|u|^{2} & \leq \frac{1}{2}\sum_{j\in \mathbb{Z}} \int_{C(j)} \frac{(x\cdot \nabla n)_{-}}{2^{j-1}n} n|u|^{2}\\
& \leq \sum_{j\in \mathbb{Z}} \frac{(x\cdot \nabla n)_{-}}{n} |||n^{1/2}u|||^{2}.\notag
\end{align}
On the other hand, let $J$ such that $2^{J-1}\leq R \leq 2^{J}$. Then by Cauchy-Schwarz inequality yields
\begin{align}
\frac{1}{R}\int_{|x|\leq R} |x||B_{\tau}||\D u||u| & \leq\frac{1}{R} \left(\int_{|x|\leq R} |\D u|^{2}\right)^{\frac{1}{2}}\left(\int_{|x|\leq R}|x|^{2}|B_{\tau}|^{2}|u|^{2} \right)^{\frac{1}{2}}\\
& \leq \frac{1}{4R}\int_{|x|\leq R}|\D u|^{2} + \sum_{j\leq J} \frac{2^{2j}|B_{\tau}|^{2}}{n(x)}|||n^{1/2}u|||^{2}\notag
\end{align}
and
\begin{align}
\int_{|x|\geq R} |B_{\tau}||\D^{\bot}u| |u| & \leq \left( \int_{|x|\geq R}\frac{|\D^{\bot}u|^{2}}{|x|}\right)^{1/2}\left(\int_{|x|\geq R} |x||B_{\tau}|^{2}|u|^{2} \right)^{1/2}\\
& \leq \frac{1}{4}\int_{|x|\geq R} \frac{|\D^{\bot}u|^{2}}{|x|} +\sum_{j\geq J} \frac{2^{2j}|B_{\tau}|^{2}}{n(x)}|||n^{1/2}u|||^{2}.\notag
\end{align}

Finally, let us analyze the $\varepsilon$ term. In this case, the a-priori estimate (\ref{b+++0}) reads as
\begin{equation}\notag
\int |\D u|^{2} \leq \int n|u|^{2} + \int |f||u|,
\end{equation}
which together with assumptions (\ref{1.7})-(\ref{1.8}) implies
\begin{equation}\notag
\int |\D u|^{2} \leq C\left(\int n_{2}|u|^{2} + \int |f||u| \right).
\end{equation}
Hence, by the same method as in \cite{Z} it follows that
\begin{equation}\label{epsilonterm2}
\varepsilon\int |\D u||u| \leq \kappa |||n^{1/2}u|||^{2} + C_{\kappa}(\varepsilon + \sup |n_{2}|)\left( N\left(\frac{f}{n^{1/2}}\right)\right)^{2}.
\end{equation}

As a consequence, plugging (\ref{fterm2})-(\ref{epsilonterm2}) into (\ref{idergy}) and taking the supremmum over $R$, by condition (\ref{1.9}) we get (\ref{1.10}), which is our claim.
\end{proof}

\begin{remark}\label{twod}
The dimension two is a special case. In this case, with the above choice of multipliers it follows that
\begin{equation}\notag
\Delta (2\varphi - \Delta\psi) \leq -\frac{C}{|x|^{3}}.
\end{equation}
Because of this singularity at zero, we cannot recover the full result (\ref{1.10}) for $d=2$ and we cannot reach the right behavior close to $0$. With some modifications in the proof (see \cite{PV1}, section 5 for more details) and assuming that $n>n_{0}>0$, then in the two dimensional case it may be proved that for $R_{0}= n_{0}^{-1/2}$ the solution satisfies
\begin{equation}\notag
|||\D u|||_{R_{0}}^{2} + |||n^{1/2}u|||_{R_{0}}^{2} + \int_{|x|\geq R_{0}} \frac{|\D^{\bot}u|^{2}}{|x|} \leq C(1+\varepsilon)\left( N_{R_{0}}\left( \frac{f}{n^{1/2}}\right)\right)^{2}.
\end{equation} 
\end{remark}

The homogeneity of the above estimate makes it compatible with the high frequencies (replace $n$ by $\mu^{2}n$). Moreover, (\ref{1.10}) allows us to get the new energy estimate. As we have already said, the estimate of the tangential component of the magnetic gradient given in (\ref{1.10}) turns out to be fundamental. In order to get it, we need the smallness assumption given in (\ref{1.9}). However, the condition (\ref{1.9}) is necessary and can not be relaxed to a Coulomb type of decay, even if smallness is added (see \cite{PV2}, Appendix for more details).

We may now state the main result of this section, which together with the above result proves Theorem \ref{Theoremenergy1}. Its interest relies on the bounds stated in Theorem \ref{Theorem1.1}.

\begin{thm}\label{Theoremenergy12}
For dimensions $d\geq 3$, we assume (\ref{nbtau})-(\ref{1.13'}) and use the notation of Theorem \ref{Theorem1.1}. Then the solution of the Helmholtz equation (\ref{**hauda}) satisfies, for $R\geq 1$ large enough
\begin{equation}\label{enerthm1}
\int_{|x|\geq R} \left|\nabla_{\omega}n_{\infty}\left(\frac{x}{|x|}\right) \right|^{2}\frac{|u|^{2}}{|x|} \leq C\left[M^{2} + \left(N\left(\frac{f}{n^{1/2}}\right)\right)^{2}\right],
\end{equation}
for some constant $C$ independent of $\varepsilon$ and $n$.
\end{thm}

\begin{proof}
The proof consists in using the basic identity (\ref{(4.3)}) with a test function that depends on the behavior of  $n(x)$ at infinity. We choose $R\geq 1$ such that (\ref{condin}) holds and define
\begin{equation}\notag
\psi_{q}(x) = q\left(\frac{|x|}{R}\right)n_{\infty}\left(\frac{x}{|x|} \right)
\end{equation}
for some non-decreasing smooth function
\begin{displaymath}
q(r) = \left\{ \begin{array}{ll}
0 & \textrm{for} \quad  r \leq 1\\
r & \textrm{for} \quad  r \geq 2.
\end{array} \right.
\end{displaymath}
Let us put $\psi_{q}$ into (\ref{(4.3)}), obtaining
\begin{align}\label{batuketabilaplacian}
&\frac{1}{2}\int \lambda \nabla\p \cdot \nabla\psi_{q} |u|^{2} = -\int \D u \cdot \D^{2}\psi_{q}\cdot\overline{\D u}\\
& - \frac{1}{2}\Re\int \nabla(\Delta\psi_{q})\cdot \D u \bar{u}+ \Im\sum_{j,k=1}^{d}\int \frac{\partial\psi_{q}}{\partial x_{k}}B_{jk}(\D)_{j}u \bar{u}\notag\\
&  - \Re\int f\nabla\psi_{q}\cdot \overline{\D u} - \frac{1}{2}\int f\Delta\psi_{q}\bar{u} + \varepsilon\int \nabla\psi_{q}\cdot \D u\bar{u}.\notag
\end{align}

We simplify the notation using $q=q\left( \frac{|x|}{R}\right)$, $n_{\infty}=n_{\infty}(\omega)$. Observe that
\begin{align}\label{partialq}
\frac{\partial \psi_{q}}{\partial x_{k}} = \frac{q' n_{\infty}}{R}\frac{x_{k}}{|x|} + \frac{q}{|x|}\frac{\partial n_{\infty}}{\partial\omega_{k}}.
\end{align} 
and
\begin{align}\label{laplaceq}
\Delta\psi_{q}=\frac{q''}{R^{2}}n_{\infty}+\frac{q'}{R}\frac{d-1}{|x|}n_{\infty}+\frac{q}{|x|^{2}}\Delta_{\omega} n_{\infty}.
\end{align}

The left hand side of the estimate (\ref{enerthm1}) will come from the term
\begin{equation}\notag
\int \lambda\nabla\p \cdot \nabla \psi_{q} |u|^{2} = \int \nabla n \cdot \nabla \psi_{q} |u|^{2}
\end{equation}
which can be written as follows
\begin{align}\label{clave}
\int\nabla n\cdot\nabla\psi_{q}|u|^{2} & = \int q\left|\frac{\partial n_{\infty}}{\partial\omega}\right|^{2}\frac{|u|^{2}}{|x|}\\
& +\int \partial_{r}n \frac{q'}{R}n_{\infty}|u|^{2}\notag\\
& + \int q|x|\nabla_{\tau}\left(n-n_{\infty}\right)\frac{\partial n_{\infty}}{\partial\omega}\frac{|u|^{2}}{|x|^{2}}\notag\\
& \equiv I_{1}+I_{2}+I_{3}.\notag
\end{align}
The first term on the right-hand side of (\ref{clave}) is the one that gives the lower bound of what we want to control. By (\ref{1.9}) and (\ref{1.12}), we get
\begin{equation}\notag
I_{2} \geq -\frac{C \Vert n_{\infty}\Vert_{L^{\infty}}}{R}|||n^{1/2}u|||^{2}.
\end{equation}
On the other hand, after integration by parts, by the diamagnetic inequality
$$
|\nabla |u|| \leq |\D u|
$$
(see \cite{LL}) and by (\ref{1.13'}), we obtain
\begin{align}
I_{3} & = -\int \frac{q}{|x|^{2}} (n-n_{\infty}) \left( \Delta_{\omega}n_{\infty}|u|^{2} +2\frac{\partial n_{\infty}}{\partial \omega}|x|\nabla |u||u|\right)\notag\\
& \leq \frac{C}{R}\Vert n_{\infty} \Vert_{C^{2}}|||n^{1/2}u|||^{2} + \kappa\int q\left|\frac{\partial n_{\infty}}{\partial \omega} \right|^{2}\frac{|u|^{2}}{|x|^{2}} + \frac{C(\kappa)}{R}|||\D u|||^{2},\notag
\end{align}
for $\kappa >0$.

Let us estimate now the remaining terms of the identity (\ref{batuketabilaplacian}). A straightforward computation gives
\begin{align}
\D u\cdot D^{2}\psi_{q}\cdot \overline{\D u} & = \frac{q''}{R^{2}}n_{\infty}|\D^{r}u|^{2} +\frac{q'}{R|x|}|\D^{\bot}u|^{2}n_{\infty}\notag\\
& +2\Re\left( \frac{q'}{R|x|}-\frac{q}{|x|^{2}}\right) \overline{\D^{r}u}\frac{\partial n_{\infty}}{\partial\omega}\cdot \D^{\bot}u\notag\\
& + \frac{q}{|x|^{2}}\D^{\bot}u\cdot D^{2}_{\omega}n_{\infty}\cdot \overline{\D^{\bot}u}.\notag
\end{align}
Thus since $q'$, $q''$ and $\left(\frac{q'}{R|x|} -\frac{q}{|x|^{2}} \right)$ are supported in the ball $\{|x|\leq 2R\}$, by the Cauchy-Schwarz inequality it follows that the absolute value of the above terms in the corresponding integral are bounded by
\begin{align}\notag
\frac{C\Vert n_{\infty}\Vert_{C^{2}}}{R} \left( \int \frac{|\nabla_{A}^{\bot}u|^{2}}{|x|} + |||\D u|||^{2}\right).
\end{align}
Moreover, by (\ref{laplaceq}) and (\ref{condin}) one can easily check that
\begin{align}\label{2.18}
Re\int \nabla(\Delta\psi_{q})\cdot \D u \bar{u}& \leq C\Vert n_{\infty} \Vert_{C^{3}}\int \left( \frac{q}{|x|^{3}}+\frac{q'}{R|x|^{2}} \right)|\D u||u|\notag\\
& \leq \frac{C\Vert n_{\infty} \Vert_{C^{3}}}{R}|||n^{1/2}u||| |||\D u|||.\notag
\end{align}
As far as the term containing the magnetic potential is concerned, first note that by (\ref{partialq}) and the fact that
$$
B_{\tau} \cdot \D u = B_{\tau} \cdot \D^{\bot} u,
$$ 
yields
\begin{equation}\label{bjk1}
\sum_{j,k=1}^{d} \frac{\partial \psi_{q}}{\partial x_{k}} B_{jk}(\D)_{j} u = \frac{q'n_{\infty}}{R}B_{\tau}\cdot \D^{\bot}u + \frac{q}{|x|}\sum_{j,k=1}^{d}\frac{\partial n_{\infty}}{\partial \omega_{k}}B_{jk}(\D)_{j}u.
\end{equation}
Thus by (\ref{nbtau}) and (\ref{assubjk}), we get
\begin{align}
\Im \sum_{j,k=1}^{d}\int \frac{\partial\psi_{q}}{\partial x_{k}}B_{jk}(\D)_{j}u\bar{u} & \leq \frac{C\Vert n_{\infty} \Vert}{R} \left( \int_{|x|\geq R}\frac{|\D^{\bot}u|^{2}}{|x|}\right)^{1/2}|||n^{1/2}u|||\notag\\
& + \frac{C\Vert n_{\infty} \Vert_{C^{1}}}{R} ||| \D u||| |||n^{1/2}u|||.\notag
\end{align}
%Similarly, by (\ref{condicionf}) and (\ref{laplaceq}), the terms containing the potential $V_{2}$ are upper bounded by
%\begin{align}
%\frac{1}{2}\int V_{2}\Delta\psi_{q}|u|^{2} + \Re\int V_{2}\nabla\psi_{q}\cdot\D u \bar{u} \leq \frac{C\Vert n_{\infty} \Vert_{C^{1}}}{R}\left(|||u|||^{2} + |||\D u|||^{2}\right).
%\end{align}
We now turn to analyze the terms containing $f$. On the one hand, by (\ref{laplaceq}), we have
\begin{equation}
\int |f| |\Delta\psi_{q}| |u| \leq \frac{C}{R^{2}}\int_{|x|>R}|f||u| \leq \frac{C}{R^{2}}N\left(\frac{f}{n^{1/2}}\right)|||n^{1/2}u|||\notag.
\end{equation}
On the other hand, from (\ref{partialq}) it follows that
\begin{align}
 \int |f| |\nabla\psi_{q}| &|\D u|  \leq \int \frac{q'}{R}|f||n_{\infty}||\D u| +\int \frac{q}{|x|}|f|\left|\frac{\partial n_{\infty}}{\partial\omega} \right||\D^{\bot}u| \notag\\
& \leq \frac{C\Vert n_{\infty}\Vert_{C^{1}}}{R}N\left(\frac{f}{n^{1/2}}\right)\left( |||\D u|||+ \left(\int_{|x| \geq R} \frac{|\D^{\bot} u|^{2}}{|x|} \right)^{1/2} \right)\notag.
\end{align}
Finally, the last term to be bounded is 
\begin{equation}\notag
\varepsilon\int \nabla\psi_{q}\cdot \D u \bar{u} \leq \frac{C\Vert n_{\infty}\Vert_{C^{1}}}{R}\varepsilon\int |\D u||u|,
\end{equation}
which can be done as in (\ref{epsilonterm2}).

Therefore, from the above inequalities, taking $\kappa$ small enough yields
\begin{align}
\int q\left|\frac{\partial n_{\infty}}{\partial\omega} \right|^{2}\frac{|u|^{2}}{|x|^{2}} & \leq \frac{C_{1}}{R} \left(|||n^{1/2}u|||^{2} + |||\D u|||^{2} + \int_{|x|\geq R} \frac{|\D^{\bot}u|^{2}}{|x|} \right)\notag\\
& + \frac{C_{2}}{R}\left(N\left(\frac{f}{n^{1/2}}\right)\right)^{2},\notag
\end{align}
which gives (\ref{enerthm1}) and the proof of the theorem is over.

\end{proof}

A combination of the above two results asserts the desired energy estimate (\ref{enerthm}) and the proof of Theorem \ref{Theoremenergy1} is completed.

\begin{remark}
From Remark \ref{twod} the same result holds for the two dimensional case.
\end{remark}

\begin{remark}
Condition (\ref{1.13'}) can be largely relaxed if, for example, $n - n_{\infty}$ is radial. It can be instead assumed the alternative conditions
\begin{equation}\label{2.2}
\left|n(x) - n_{\infty}\left(\frac{x}{|x|} \right) \right| \leq n\frac{\Gamma}{|x|^{\delta}} \quad \textrm{for} \quad |x|>R_{0}, \Gamma>0, \delta >0, R_{0}>1
\end{equation}
and that there exist $\tilde{\beta}<1$, $\delta >0$ and $\tilde{\Gamma}>0$ such that
\begin{equation}\notag
\left(|x|\nabla^{\bot}(n-n_{\infty}) \cdot \frac{\partial n_{\infty}}{\partial \omega}\right)_{-} \leq \tilde{\beta}\left|\frac{\partial n_{\infty}}{\partial \omega} \right|^{2} + n(x)\frac{\tilde{\Gamma}}{|x|^{\delta}}.
\end{equation}
In particular, when $n - n_{\infty}$ is radial then (\ref{2.2}) is sufficient. 
\end{remark}

\begin{remark}
Note that in order to prove the energy estimate we impose conditions in each component of the magnetic field $B_{jk}$ and not in the tangential component of $B$, as in the first result. This is due to the fact that the test function chosen in the proof of Theorem \ref{Theoremenergy1} is not radial (see (\ref{partialq}) and (\ref{bjk1}) above). 
\end{remark}

\section{Limiting absorption principle. Proof of Theorem \ref{d>3}}\label{LAPenergy}

Our next goal is to prove Theorem \ref{d>3} which asserts the limiting absorption principle for the equation (\ref{**s}), following \cite{S} and \cite{Z}. In addition, the result will be true for short range electric potentials that can have singularities at the origin and more importantly, critical singularities at the origin for the magnetic field can be considered. 

To do this, let us consider the electromagnetic Helmholtz equation
\begin{equation}\label{2.10}
\D^{2}u + \lambda(1+\p)u + Q u + i\varepsilon u = f.
\end{equation}
We first prove the corresponding Sommerfeld radiation condition and a-priori estimates for the solution $u\in H^{1}_{A}(\Rd)$ of this equation for $\lambda \in [\lambda_{0}, \lambda_{1}]$ with $0<\lambda_{0} < \lambda_{1} < \infty$ and $\varepsilon >0$. We next turn to show the uniqueness result related to this equation. Indeed, we will see that if $u$ satisfies (\ref{2.10}) with $\varepsilon = 0$ and $f=0$, then $u\equiv 0$. Consequently, we will be in a position to construct the unique solution of the equation (\ref{**s}) with the radiation condition (\ref{radiacion}) at infinity as the limit of the solution of the equation (\ref{2.10}) when $\varepsilon \to 0$ in some sense. The detailed proof of this construction is given in \cite{Z}, see subsection 2.4. Thus we will omit it.

%\vspace{0.3cm}
%We begin by proving that the Sommerfeld radiation condition holds if the Agmon-H\"ormander norm of the solution of the electromagnetic Helmholtz equation (\ref{2.10}) is %bounded. Making use of this inequality, we deduce the a-priori estimates when $\lambda\in[\lambda_{0}, \lambda_{1}]$ by a compactness argument already used in \cite{S} and %\cite{Z}. Finally, we state and prove the uniqueness of solution of the equation (\ref{**s}). 

Since the proofs are adapted from the ones of the main results of \cite{Z}, we will mainly focus on the analysis of the new terms, that is to say, $\p$.

\subsection{Sommerfeld radiation condition}

We begin by proving the Sommerfeld condition in terms of the Agmon-H\"ormander norm of the solution. This result may be proved in much the same way as Proposition 2.6 of \cite{Z}.

\begin{pro}\label{propositionradiationSaito}
For dimensions $d\geq 3$, let $\lambda_{0} > 0$, $\varepsilon > 0$, $f\in L^{2}_{\frac{1+\delta}{2}}$ and assume that (\ref{condicionf}) holds. Let $K$ be a solution of the eikonal equation (\ref{eikonal}). Then, there exists a positive constant $C = C(\lambda_{0})$ such that for $\lambda \geq  \lambda_{0}$ and $C^{*}$ small enough, any solution $u\in H^{1}_{A}(\Rd)$ of the equation (\ref{2.10}) satisfies for all $R_{1} \geq r_{0}$
\begin{align} \label{1chap3}
&\int_{K \geq R_{1}}|\D u - i\sqrt{\lambda} \nabla K u|^{2}\left(  \frac{1}{(1+K)^{1-\delta}} + \varepsilon (1+K)^{\delta} \right)\\
&+(1-\delta)\int_{K \geq R_{1}} \frac{|\nabla K|^{2}|\D u|^{2}- |\nabla K\cdot \D u|^{2}}{(1+K)^{1-\delta}}\notag\\
& \leq C(1+\varepsilon)\left(|||u|||_{1}^{2} + (N_{1}(f))^{2} + \int_{K \geq R_{1}} (1+K)^{1+\delta}|f|^{2}\right).\notag
\end{align}
\end{pro}

\begin{proof}
The proof will be divided into three steps and it consists in the construction of the Sommerfeld terms which contain the square
\begin{equation}\label{somsquare}
|\D u - i \sqrt{\lambda} \nabla K u|^{2} = |\D u|^{2} + \lambda|\nabla K|^{2}|u|^{2} - 2\Im \sqrt{\lambda}\nabla K \cdot \D u \bar{u}.
\end{equation}
We use the identities proved in Lemma \ref{appendix1} and Lemma \ref{appendix2}. In this case, one must choose the multipliers depending on the solution of the eikonal equation. By abuse of notation, we write $\nabla \psi$ instead of a vector field $E$.

Let $R_{1}\geq r_{0}$. We take a cut off function $\theta \in C^{\infty}(\R)$ such that $0\leq \theta \leq 1$, $d\theta/dr \geq 0$ with 
\begin{equation}\notag
\theta(r) = \left\{ \begin{array}{ll}
1 & \textrm{if $r \geq R_{1}+1$}\\
0 & \textrm{if $r \leq R_{1}$},
\end{array} \right.
\end{equation}
and set $\theta(K)=\theta(K(x,C^{*}))$. We define $\Psi:\R \to \R$ such that
$$
\Psi'(r) = (1+r)^{\delta}, \quad \quad \quad 0<\delta <1
$$
and we set $\Psi(K)=\Psi(K(x,C^{*}))$.

\textbf{Step 1.} Our first goal is to obtain the term
$$
\int |\nabla K|^{2} (|\D u |^{2} + \lambda |\nabla K|^{2}|u|^{2}) \frac{\theta(K)}{(1+K)^{1-\delta}}.
$$

Let us first compute
$$
(\ref{(4.3)}) + (\ref{(4.11)}), 
$$
with the following choice of the multipliers
\begin{align}
& E=\nabla\psi= \Psi'(K)\nabla K \theta(K)\notag\\
& \varphi(x) = \frac{\delta|\nabla K|^{2}}{2(1+K)^{1-\delta}}\theta(K)\notag,
\end{align}
respectively. Let us analyze all the terms of the resulting identity by the same method as in the proof of Proposition 2.6 in \cite{Z}. In what follows, $\kappa$ denotes an arbitrary positive small constant and we use the same letter $C$ for any positive constant.

On the one hand, by (\ref{segunda}) and the fact that $\theta'$ is nonnegative, we get
\begin{align}
&\int \D u\cdot D^{2}\psi\cdot \overline{\D u} - \int \varphi |\D u|^{2}   > \frac{\delta}{2} \int_{\Rd} \frac{|\nabla K|^{2}|\D u|^{2}}{(1+K)^{1-\delta}}\theta(K)\notag\\
& +\int \theta(K) \left( \frac{(1+K)^{\delta}}{K}-\frac{\delta}{(1+K)^{1-\delta}} \right)\{|\nabla K|^{2}|\D u|^{2}-|\nabla K \cdot \D u|^{2}\}\notag\\
& + \int \frac{(1+K)^{\delta}}{K}\sum_{k,j=1}^{d}(\D)_{k}u F_{kj} \overline{(\D)_{j}u} \theta(K)\notag\\
& \equiv I_{1} + I_{2} +I_{3},\notag
\end{align}
where 
\begin{equation}\notag
I_{2} \geq (1-\delta)\int \frac{\theta(K)}{(1+K)^{1-\delta}}\{|\nabla K|^{2}|\D u|^{2}- |\nabla K\cdot \D u|^{2}\}.
\end{equation}
On the other hand, observe that in order to get the term related to $|u|^{2}$ of the Sommerfeld square $|\D u - i\lambda^{1/2}\nabla K u|^{2}$, we need to use the eikonal equation (\ref{eikonal}). Indeed, we have
\begin{align}\notag
\int \varphi \lambda(1+\p) |u|^{2} &= \frac{\delta}{2}\int \frac{ |\nabla K|^{2} \lambda|\nabla K|^{2} |u|^{2}}{(1+K)^{1-\delta}}\theta(K).
\end{align}
Moreover, by the eikonal equation $\p(x)=|\nabla K|^{2}-1$ and (\ref{segunda}), it follows that
\begin{equation}\label{p}
\frac{\partial\p}{\partial x_{k}} = 2\sum_{j=1}^{d} \frac{1}{K}F_{kj}\frac{\partial K}{\partial x_{j}} \quad \textrm{for all $k=1, \ldots, d$}.  
\end{equation}
Thus the other term involving the potential $\p$ gives
\begin{align}
-\frac{\lambda}{2}\int \nabla\p\cdot\nabla\psi|u|^{2} &= -\lambda\sum_{k,j=1}^{d} \int \frac{(1+K)^{\delta}}{K}\frac{\partial K}{\partial x_{k}}F_{kj}\frac{\partial K}{\partial x_{j}}|u|^{2}\theta(K)\equiv I_{4}.\notag
\end{align}
Let us treat now the terms containing the magnetic field $B$ and the potential $Q$. Since $c\leq |\nabla K|^{2} \leq \tilde{c}$ for some $c,\tilde{c}>0$, by (\ref{condicionf}) and Cauchy-Schwarz inequality we get
\begin{align}
\sum_{k,m =1}^{d} \int \frac{\partial \psi}{\partial x_{k}} B_{km} u& \overline{(\D)_{m}u}  \leq C\int |B_{km}| |\D u| |u|(1+K)^{\delta}\theta(K)\notag\\
& \leq \kappa\int |\nabla K|^{2}|\D u-i\sqrt{\lambda}\nabla K u|^{2}\frac{\theta(K)}{(1+K)^{1-\delta}}\notag\\
& +C_{\kappa}(\sqrt{\lambda}+1)|||u|||_{1}^{2}.\notag
\end{align}
Similarly, by (\ref{condicionf}) we have
\begin{align} \Re  \int Q \nabla \psi \cdot \D u \bar{u} 
& \leq \kappa\int \frac{|\nabla K|^{2} |\D u -i \sqrt{\lambda}\nabla K u|^{2}\theta(K)}{(1+K)^{1-\delta}}\notag\\
& + C_{\kappa}(\sqrt{\lambda}+1)|||u|||_{1}^{2}.\notag
\end{align}
In addition, since 
$$
\Delta \psi = \Psi''(K)|\nabla K|^{2}\theta(K) + \Psi'(K)\Delta K\theta(K) + \Psi'(K)|\nabla K|^{2} \theta'(K),
$$
by (\ref{g}), (\ref{propietateg}) it may be concluded that
\begin{align}\notag
-\int \varphi Q |u|^{2} +  \frac{1}{2}\int Q \Delta \psi |u|^{2} \leq C|||u|||_{1}^{2}.
\end{align} 
As a consequence, we get the inequality 
\begin{align}\label{step1}
&\frac{\delta}{2} \int |\nabla K|^{2} ( |\D u|^{2} + \lambda|\nabla K|^{2} |u|^{2}) \frac{\theta(K)}{(1+K)^{1-\delta}}\\ 
& +(1-\delta)\int \frac{\theta(K)}{(1+K)^{1-\delta}}\{|\nabla K|^{2}|\D u|^{2}- |\nabla K\cdot \D u|^{2}\}\notag\\
&- \varepsilon Im \int\theta(K)\Psi'(K)\nabla K \cdot \D u \bar{u}  \leq -I_{3} + I_{4}\notag\\
& +2\kappa\int \frac{|\nabla K|^{2} |\D u -i \sqrt{\lambda}\nabla K u|^{2}\theta(K)}{(1+K)^{1-\delta}} +C(\sqrt{\lambda} +1)|||u|||_{1}^{2}\notag\\
& -\Re \int f \left(\Psi'(K)\nabla K\cdot\overline{\D u} +\frac{1}{2}\Psi'(K)\Delta K \right)\theta(K)\bar{u}\notag\\
& -\frac{\Re}{2}\int \Psi'(K)|\nabla K|^{2}\theta'(K)\bar{u}.\notag
\end{align}

\textbf{Step 2.} In order to obtain the desired square, the next step is to get the cross term
$$
-2\Im \int \sqrt{\lambda}|\nabla K|^{2} \nabla K \cdot \D u \bar{u} \frac{\theta(K)}{(1+K)^{1-\delta}}.
$$

Let us add to the above inequality (\ref{step1}) the identity $(\ref{(4.2)})$ with the choice of a test function 
\begin{equation}\notag
\varphi(x)= \sqrt{\lambda}|\nabla K|^{2}(1+K)^{\delta}\theta(K).
\end{equation}
Hence, it follows that
\begin{align}
\Im \int \nabla \varphi \cdot \D u \bar{u} & = \delta \Im\sqrt{\lambda} \int \frac{\theta(K)}{(1+K)^{1-\delta}}|\nabla K|^{2}\nabla K\cdot \D u  \bar{u}\notag\\
& + \sqrt{\lambda}\Im \int |\nabla K|^{2} \theta^{'}(K)(1+K)^{\delta} \nabla K \cdot \D u \bar{u}\notag\\
& +2\sqrt{\lambda}\Im\int \frac{(1+K)^{\delta}}{K}\sum_{k,j=1}^{d}(\D)_{k}uF_{kj}\frac{\partial K}{\partial x_{j}}\bar{u}\theta(K)\notag\\
& \equiv I_{5} + I_{6} + I_{7}.\notag
\end{align}
The term $I_{5}$ is used to complete the square $|\D u - i\sqrt{\lambda}\nabla K u|^{2}$; $I_{6}$ can be upper bounded by 
\begin{equation}\notag
\kappa\int |\nabla K|^{2}|\D u-i\sqrt{\lambda}\nabla K u|^{2}\frac{\theta(K)}{(1+K)^{1-\delta}} + C_{\kappa,R_{1}}(1+\lambda)|||u|||_{1}^{2}.
\end{equation}
In addition, denoting $(D_{K})_{i}u=(\D)_{i}u-i\sqrt{\lambda}\frac{\partial K}{\partial x_{i}}u$, by (\ref{F}) it may be concluded that
\begin{align}
-I_{3} + I_{4} + I_{7} & = -\int \frac{(1+K)^{\delta}}{K}\sum_{k,j=1}^{d}(D_{K})_{k}uF_{kj}\overline{(D_{K})_{j}u}\theta(K)\notag\\
& \leq C C^{*}\int \frac{|\nabla K|^{2} |\D u -i \sqrt{\lambda}\nabla K u|^{2}\theta(K)}{(1+K)^{1-\delta}}.\notag
\end{align}

Therefore, we deduce
\begin{align}
&\frac{\delta}{2}\int |\nabla K|^{2}|\D u -i\sqrt{\lambda}\nabla K u|^{2}\frac{\theta(K)}{(1+K)^{1-\delta}}\notag\\
& +(1-\delta)\int  \frac{\theta(K)}{(1+K)^{1-\delta}}\{|\nabla K|^{2}|\D u|^{2}- |\nabla K\cdot \D u|^{2}\} \notag\\
& +\varepsilon\sqrt{\lambda}\int |\nabla K|^{2}(1+K)^{\delta}|u|^{2}\theta(K) - \varepsilon\Im\int \theta(K) (1+K)^{\delta}\nabla K \cdot\nabla_{A}u\bar{u}\notag\\
& \leq C(\lambda +1)|||u|||^{2} + (3\kappa+CC^{*})\int \frac{|\nabla K|^{2}|\D u -i\sqrt{\lambda}\nabla K u|^{2}\theta(K)}{(1+K)^{1-\delta}}\notag\\
& -\Re\int f(1+K)^{\delta}\nabla K \cdot (\overline{\D u} + i\lambda^{1/2}\nabla K\bar{u})\theta(K)\notag\\
& -\frac{\Re}{2}\int f\Psi'(K)(\Delta K \theta(K)+|\nabla K|^{2}\theta'(K))\bar{u}.\notag
\end{align}

\textbf{Step 3.} In order to conclude the proof of Proposition \ref{propositionradiationSaito}, we build the Sommerfeld square (\ref{somsquare}) for the $\varepsilon$ term.

Let us subtract the identity $(\ref{(4.11)})$ multiplied by $\varepsilon$ to the above inequality choosing the test function
\begin{displaymath}
\varphi(x)=\frac{1}{2\sqrt{\lambda}}\Psi'(K)\theta(K),
\end{displaymath}
so that we get 
\begin{equation}\notag
\frac{\varepsilon}{\lambda^{1/2}}\int |\nabla K|^{2} (1+K)^{\delta}|\D u - i\lambda^{1/2}\nabla K u|^{2}\theta(K).
\end{equation}

In order to complete the estimate, by integration by parts and the a-priori estimate (\ref{a+++0}), we have
\begin{align}
\varepsilon\Re \int\nabla\varphi\cdot \D u\bar{u} & =\frac{\varepsilon}{2}\int \Delta\varphi |u|^{2}\notag\\
&\leq C\varepsilon\int |u|^{2} \leq CN_{1}(f)|||u|||_{1}.\notag
\end{align}
Furthermore, by (\ref{condicionf}) we deduce
\begin{align}
\varepsilon\int \varphi Q |u|^{2} &\leq \frac{C\varepsilon}{\lambda^{1/2}}\int_{|x|\geq r_{0}} \frac{|u|^{2}}{(1+|x|)^{1+\mu-\delta}}\notag\\
& \leq C \varepsilon|||u|||_{1}^{2}.\notag
\end{align}
Finally, let us estimate the terms containing $f$. On the one hand, we have
\begin{align}
-\Re\int f(1+K)^{\delta}&\nabla K \cdot (\overline{\D u} +i\lambda^{1/2}\nabla K \bar{u})\theta(K)\notag\\
&\leq \kappa\int  |\nabla K|^{2} |\D u - i\lambda^{1/2}\nabla K u|^{2}\frac{\theta(K)}{(1+K)^{1-\delta}}\notag\\
& + C(\kappa)\int (1+K)^{1+\delta} |f|^{2}\theta(K)\notag.
\end{align}
By (\ref{propietateg}), we get
\begin{align}
-\frac{\Re}{2}\int &\Psi'(K) (\Delta K \theta(K) + |\nabla K|^{2}\theta'(K))f\bar{u}\notag\\
& \leq C\left( |||u|||_{1}^{2} + \int (1+K)^{1+\delta}|f|^{2}\theta(K)\right).\notag
\end{align}
By the a-priori estimate (\ref{a+++0}), yields
\begin{align}
-\frac{\varepsilon}{2\sqrt{\lambda}}\Re &\int (1+K)^{\delta}f\bar{u}\theta(K)\notag\\
& \leq \left(\frac{4\varepsilon}{\lambda}\int|f|^{2}(1+K)^{1+\delta}\theta(K) \right)^{1/2}\left(\varepsilon\int |u|^{2} \right)^{1/2}\notag\\
& \leq C\left( \varepsilon\int |f|^{2}(1+K)^{1+\delta}\theta(K) + |||u|||_{1}^{2} + (N_{1}(f))^{2}\right).\notag
\end{align}

Consequently, taking $\kappa > 0$ and $C^{*}$ small enough, we obtain (\ref{1chap3}) and the proof is complete.
\end{proof}

\begin{cor}
Under the assumption of Proposition \ref{propositionradiationSaito}, the solution $u\in H^{1}_{A}(\Rd)$ of the Helmholtz equation (\ref{2.10}) satisfies 
\begin{align}\label{radiacion1}
&\int_{|x|\geq r_{0}} \frac{|\D u -i\lambda^{1/2}\nabla K u|^{2}}{(1+|x|)^{1-\delta}} + \varepsilon\int_{|x|\geq r_{0}} (1+|x|)^{\delta}|\D u -i\lambda^{1/2}\nabla K u|^{2}\\
& \leq C(1+\varepsilon)\left(|||u|||_{1}^{2} + (N_{1}(f))^{2} + \int_{|x|\geq r_{0}} (1+|x|)^{1+\delta}|f|^{2}\right),\notag
\end{align}
for $\lambda\geq \lambda_{0}$, $C^{*}$ small enough and $C=C(\lambda_{0})$.
\end{cor}

\begin{proof}
We need only take $R_{1}=c_{0}r_{0}$ with $c_{0}, r_{0}$ given in section \ref{eikonalequation} and use (\ref{g}).
\end{proof}

\subsection{A priori estimates for $\lambda\in [\lambda_{0}, \lambda_{1}]$}
Using the previous result, we are now in a position to prove the a-priori estimates for the frequency $\lambda$ varying in a compact set. We will deduce them by a compactness argument already used in \cite{S} and \cite{Z}.

\begin{pro}\label{proposition1}
For $d\geq 3$, under the hypotheses of Proposition \ref{propositionradiationSaito}, let $\lambda_{0} > 0$, $\lambda \in [\lambda_{0}, \lambda_{1}]$, with $\lambda_{1} >\lambda_{0}$ and $\varepsilon \in (0, \varepsilon_{1})$. Then, the solution $u\in H^{1}_{A}(\Rd)$ of the Helmholtz equation (\ref{2.10}) satisfies
\begin{equation}\label{landapeque–o}
\lambda|||u|||_{1}^{2} + |||\D u|||_{1}^{2}  \leq C (1+\varepsilon) (N_{1}(f))^{2},
\end{equation}
where $C=C(\lambda_{0}, \varepsilon_{1})$.
\end{pro}

\begin{proof}
The proof is a combination of the proof of Proposition 3.1 in \cite{S} and the proof of Proposition 2.10 in \cite{Z}.

Let $B_{T}$ be the interior of the closed surface $\Sigma_{T}=\{x: K(x,C^{*})=T\}$ with $T> r_{0}$ and $C^{*}<C_{0}$, where $r_{0}$ and $C_{0}$ are given constants related to the assumptions of the potentials and the solution to the eikonal equation, respectively. Let us multiply the equation (\ref{2.10}) by $\bar{u}$, integrate over $B_{T}$ and take the imaginary part, obtaining
\begin{equation}
\Im \int_{\Sigma_{T}} \frac{\nabla K}{|\nabla K|}\cdot \D u\bar{u} +\varepsilon\int_{B_{T}}|u|^{2}=\Im\int_{B_{T}}f\bar{u}\notag.
\end{equation}
From this it follows that 
\begin{equation}\label{3.3}
2\sqrt{\lambda}\Im \int_{\Sigma_{T}} \frac{\nabla K}{|\nabla K|}\cdot \D u\bar{u} \leq 2\sqrt{\lambda}\Im\int_{B_{T}}f\bar{u}.
\end{equation}
Let us integrate now the identity
\begin{equation}\label{3.5}
\frac{|\D u|^{2}}{|\nabla K|} + \lambda|\nabla K| |u|^{2} = \frac{1}{|\nabla K|}|\D u -i\sqrt{\lambda}\nabla K u|^{2} + 2 \Im \sqrt{\lambda} \frac{\nabla K}{|\nabla K|} \cdot \D u \bar{u}\notag
\end{equation}
over the surface $\Sigma_{T}$. Then by (\ref{3.3}) we get
\begin{align}\label{3.6}
\int_{\Sigma_{T}}\left(\frac{|\D u|^{2}}{|\nabla K|} + \lambda|\nabla K| |u|^{2}\right) & \leq \int_{\Sigma_{T}} \frac{1}{|\nabla K|} |\D u -i\sqrt{\lambda}\nabla K u|^{2}\\
&  +2\sqrt{\lambda}N_{1}(f)|||u|||_{1}.\notag
\end{align}
\noindent
Let $R>\frac{\rho c_{0}}{c_{1}}$, where $\rho \geq r_{0}$, being $c_{0}$, $c_{1}$ as in (\ref{g}). Let us multiply both sides of (\ref{3.6}) by $\frac{1}{R}$ and integrate from $\rho c_{0}$ to $Rc_{1}$ with respect to $T$. Hence, as $|\nabla K|^{2}$ is lower bounded by a positive constant we have
\begin{align}
\frac{1}{R}\int_{\rho c_{0} \leq K \leq Rc_{1}} (\lambda|u|^{2} + |\D u|^{2}) & \leq \frac{1}{R}\int_{\rho c_{0}\leq K \leq Rc_{1}} |\D u - i\lambda^{1/2}\nabla K u|^{2}\notag\\
& + C\sqrt{\lambda}N(f)|||u|||.\notag
\end{align}
On the other hand, observe that since $K=|x|g$ and $c_{0}\leq g \leq c_{1}$, yields
$$
\{\rho\leq |x|\leq R \} \subset \{\rho c_{0} \leq K \leq Rc_{1} \} \subset \left\{ \frac{\rho c_{0}}{c_{1}} \leq |x| \leq \frac{Rc_{1}}{c_{0}} \right\}. 
$$
Consequently, denoting $j_{0}$ and $j_{1}$ by $2^{j_{0}-1} \leq \frac{\rho c_{0}}{c_{1}} \leq 2^{j_{0}}$ and $2^{j_{1}-1} \leq \frac{Rc_{1}}{c_{0}} \leq 2^{j_{1}}$, respectively, we deduce
\begin{align}\notag
\frac{1}{R}\int_{\rho\leq |x| \leq R} (\lambda|u|^{2} + |\D u|^{2})  & \leq \frac{1}{R}\sum_{j=j_{0}}^{j_{1}} \int_{C(j)}|\D u -i \lambda^{1/2}\nabla K u|^{2}\\
& + \kappa\lambda|||u|||_{1}^{2} + C(\kappa)(N_{1}(f))^{2}.\notag
\end{align}

Now, note that we are in the same position as in (2.42) of the proof of Proposition 2.10, \cite{Z}. Therefore, by (\ref{radiacion}), repeating the same reasoning to this case, it may be concluded that for $R>1$
\begin{equation}\notag
\frac{1}{R}\int_{|x|\leq R} (\lambda|u|^{2} + |\D u|^{2}) \leq \frac{\lambda}{2}|||u|||_{1}^{2} + C(1+\varepsilon)(N_{1}(f))^{2}.
\end{equation}
Thus taking the supremum over $R$, the proposition follows.
\end{proof}

\subsection{Uniqueness result}

This paragraph deals with the uniqueness of solution of the equation (\ref{**s}). Let us consider the homogeneous Helmholtz equation
\begin{equation}\label{homoSaito}
\D^{2}u + \lambda(1+\p)u + Q u =0.
\end{equation}
Then we formulate the uniqueness theorem as follows.

\begin{thm}\label{unicidad2}
Let $d \geq 3$, $\lambda_{0}>0$ and assume (\ref{condicionf}), (\ref{gaugesaito}). Let $u$ be a solution of the equation (\ref{homoSaito}) with $u,\D u \in L^{2}_{loc}$ such that 
\begin{equation}
\liminf \int_{|x|= r} (|\D u|^{2} + \lambda|u|^{2}) d\sigma(x) \to 0, \quad \textrm{as} \quad r \to \infty, \label{1.13}
\end{equation}
for $\lambda\geq \lambda_{0}$. Then $u \equiv 0$. 
\\
Moreover, if for some $\delta>0$ the condition
\begin{equation}\label{somunic}
\int_{|x|\geq 1} |\D u - i\lambda^{1/2}\nabla Ku|^{2}\frac{1}{(1+|x|)^{1-\delta}} < \infty
\end{equation}
is satisfied, then (\ref{1.13}) holds.
\end{thm}

\begin{proof}
The proof follows by the same method as in the proof of Theorem 1.5 in \cite{Z}. Although the analysis of the terms related to $\p$ are slightly different, the same conclusion can be drawn for this case. We give only the main ideas. For a fuller treatment we refer the reader to \cite{Z}.

The first step is to show
\begin{equation}\label{1.1}
\int_{|x| > R} (|\D u|^{2} + |u|^{2}) \leq \frac{C}{R^{2}} \int_{\frac{R}{2} \leq |x| \leq R} |u|^{2}.
\end{equation}
For this purpose, we multiply the equation (\ref{homoSaito}) by the combination of the symmetric and the antisymmetric multipliers $\nabla\psi \cdot \overline{\D u} + \frac{1}{2} \Delta\psi\bar{u} + \varphi\bar{u}$ and we integrate it on the ball $\{|x| < R_{1} \}$ for some $R_{1} > R>r_{0}$. Then, we define a cut off function $\theta$ with 
\begin{displaymath}
\theta(r) = \left\{ \begin{array}{ll}
1 & \textrm{if $r \geq 1$}\\
0 & \textrm{if $r < \frac{1}{2}$}
\end{array} \right.
\end{displaymath}
\noindent
and $\theta' \geq 0$ for all $r$. Set  $\theta_{R}(x) = \theta \left(\frac{|x|}{R} \right)$ and for $R >r_{0}\geq1$ choose the multipliers as follows
\begin{equation}\notag
\nabla\psi(x)= \frac{x}{R}\theta_{R}(x)
\end{equation}
and
\begin{equation}\notag
\varphi(x)=\frac{1}{2R}\theta_{R}(x).
\end{equation}
We do the computations for $R$ large enough and we pass to the limit in $R_{1}$.

The next goal is to prove that for $R > 2r_{0} \geq 1$ and any $m\geq 0$, then
\begin{equation}\notag
\int_{|x| > R} |x|^{m} (|\D u|^{2} + u|^{2}) < +\infty.
\end{equation}
This can be easily shown by induction.

Our last claim is to prove the exponential decay of the solution. We define the test functions
\begin{align}
& \nabla\psi(x)= |x|^{m+1}\frac{x}{|x|}\theta_{R}(x),\notag\\
& \varphi(x) = \frac{1}{2}|x|^{m}\theta_{R}(x),\notag
\end{align}
for $R >2r_{0}\geq 1$ and we put them into the identities (\ref{(4.3)}) and  (\ref{(4.11)}), respectively. We add both equalities and analysis similar of that in \cite{Z}, shows that for $R$ large enough and for any $t\geq 1$, $0<\delta<2/3$, $\lambda\geq\lambda_{0}$, it follows that
\begin{equation}\notag
\int_{|x|>2R} |u|^{2}  \leq C_{\delta}e^{-tR^{\delta}}\left(1+\lambda+\frac{t^{2}}{R}\right),
\end{equation}
being $C_{\delta}$ independent of $t$. Thus letting $t \to \infty$, we obtain that $u=0$ almost everywhere in $\{|x|>2R\}$. The unique continuation property (\cite{R}) implies then $u=0$ almost everywhere in $\Rd$.

In order to deduce (\ref{1.13}) from (\ref{somunic}), first observe that solutions of (\ref{homoSaito}) satisfy  
$$
\Im \int_{\Sigma_{T}} \frac{\nabla K}{|\nabla K|}\cdot \D u\bar{u} =0,
$$
just multiplying the equation by $\bar{u}$ and integrating over $B_{T}$, the inside of the closed surface $\Sigma_{T}=\{x: K(x,C^{*})=T\}$. Hence, we have  
$$\int_{\Sigma_{T}} (|\D u|^{2} + \lambda|\nabla K|^{2}|u|^{2}) d\sigma(x) = \int_{\Sigma_{T}} |\D u - i\sqrt{\lambda} \nabla Ku|^{2}d\sigma(x),$$ 
which together with (\ref{somunic}) gives (\ref{1.13}). 

\end{proof}

\section{Explicit radiation condition. Proof of Theorem \ref{teoremaexplicit}}\label{sectionexplicit}

This section establishes the relation between the energy estimate (\ref{enerthm}) and the Sommerfeld condition (\ref{radiacion}). We will see that when the variable index of refraction has the form $n(x)=\lambda(1+ \p(x))$ and an angular dependency like $n(x) \to n_{\infty}\left(\frac{x}{|x|}\right)$ as $|x|\to \infty$, then the Sommerfeld condition (\ref{radiacion}) at infinity still holds under the explicit form (\ref{explicit}).

\begin{proof}[Proof of Theorem \ref{teoremaexplicit}]

The proof follows \cite{PV2}. Let us first recall the tangential estimate
\begin{equation}\label{tangential}
\int \frac{|\D^{\bot}u|^{2}}{|x|} \leq C(N(f))^{2}
\end{equation}
proved in Theorem \ref{Theorem1.1} above and observe that (\ref{radiacion}) provides
\begin{equation}\label{barrixa}
\int |\D u-i\lambda^{1/2}\nabla K u|^{2}\frac{1}{1+|x|} \leq C \int (1+|x|)^{1+\delta}|f|^{2}dx.
\end{equation}
Hence, just looking at the tangential part of the above inequality, by (\ref{tangential}) it follows easily that
\begin{align}\label{t1}
\int_{|x| \geq r_{0}} \lambda|\nabla^{\bot} K u|^{2} \frac{1}{1+|x|} & \leq C\int \frac{|\D^{\bot}u|^{2}}{1+|x|} + C\int (1+|x|)^{1+\delta} |f|^{2}\\
& \leq C \int (1+|x|)^{1+\delta}|f^{2}|.\notag
\end{align}

Furthermore, since $n=\lambda(1+\p)$, from the eikonal equation (\ref{eikonal}) we have
\begin{equation}\notag
n - \lambda |\partial_{r} K|^{2}=|\lambda^{1/2}\nabla^{\bot} K|^{2}.
\end{equation}
Now, according to the properties (\ref{propietateg}) related to $\nabla K$, it is easy to see that  \hbox{$\partial_{r} K = g(x) + O(C^{*}) > 0$.} Thus we obtain
\begin{equation}\label{t2}
|\lambda^{1/2}\partial_{r} K - n^{1/2}|= \frac{|\lambda^{1/2}\nabla^{\bot} K|^{2}}{|\lambda^{1/2}\partial_{r}K + n^{1/2}|} \leq C|\lambda^{1/2}\nabla^{\bot} K|^{2}.
\end{equation}
In addition, looking at the radial part in (\ref{barrixa}) we have
\begin{equation}\label{t3}
\int_{|x| \geq r_{0}} |\D^{r}u-i\lambda^{1/2}\partial_{r}K u|^{2} \frac{1}{1+|x|} \leq C\int (1+|x|)^{1+\delta}|f|^{2}.
\end{equation}
Consequently, by (\ref{t1}), (\ref{t2}), (\ref{t3}) and the fact that
\begin{align}\notag
\left|\D u - in^{1/2}\frac{x}{|x|}u\right|^{2}  &\leq |\D^{r}u - i\sqrt{\lambda}\partial_{r}K u|^{2} + |\sqrt{\lambda}\partial_{r}Ku  - n^{\frac{1}{2}}u|^{2} + |\D^{\bot}u|^{2}, \notag
\end{align} 
we get (\ref{nradi}) which is our first claim.

Finally, assuming $|n-n_{\infty}| \leq C(1+|x|)^{-\delta}$ and using (\ref{1.10}) we conclude that 
\begin{equation}\notag
\int \left|\D u - in_{\infty}^{1/2}\frac{x}{|x|}u \right|^{2} \frac{1}{1+|x|} \leq C\int |f|^{2}(1+|x|)^{1+\delta} 
\end{equation}
and the proof is complete.

\end{proof}

\section{Appendix}

In this section we state the key equalities that have been used in the proofs of the main results of this work. 

These integral identities are obtained by the standard technique of Morawetz multipliers, using integration by parts (see \cite{F}, Lemma 2.1. and \cite{PV1}, Lemma 2.1). In order to carry out the integration by parts argument below, we need some regularity in the solution $u$. In general, it is enough to know that $u\in H^{1}_{A}(\Rd)$. See \cite{Z}, Appendix for more details.

Let us consider the electromagnetic Helmholtz equation
\begin{equation}\label{2.1}
(\nabla + ib)^{2} u + \lambda(1+\p) u + Q u + i\varepsilon u = f, \quad \quad \lambda , \varepsilon >0.
\end{equation}

\begin{lem}\label{appendix1} Let $\varphi: \Rd \to \R$ be regular enough. Then, the solution $u\in H^{1}_{A}(\Rd)$ of the Helmholtz equation (\ref{2.1}) satisfies
\begin{align}\label{(4.11)}
& \int  \varphi\lambda |u|^{2} - \int \varphi |\D u|^{2}  + \int \varphi (\p + Q) |u|^{2}- \Re \int \nabla \varphi \cdot \D u \bar{u}=  \Re\int  \varphi f\bar{u},
\end{align}
\begin{equation}
\varepsilon \int \varphi |u|^{2} - \Im \int \nabla \varphi\cdot \D u\bar{u} = \Im\int \varphi f\bar{u}.\label{(4.2)}
\end{equation}
\end{lem}

\begin{remark}
Note that if we take $\varphi =1$, then we obtain the following a-priori estimates 
\begin{align}\label{a+++0}
\varepsilon \int  |u|^{2} \leq \int |f| |u|
\end{align}
\begin{align}\label{b+++0}
\int |\D u|^{2} & \leq \int (\lambda + \p + Q)|u|^{2}  + \int |f|||u|,
\end{align}
that have been very useful throughout the paper.
\end{remark}

\begin{lem}\label{appendix2}
Let $\psi: \mathbb{R}^d \longmapsto \mathbb{R}$ be regular enough. Then, any solution $u\in H^{1}_{A}(\Rd)$ of the equation (\ref{2.1}) satisfies

\begin{align}\label{(4.3)}
&\int \D u\cdot D^{2} \psi \cdot \overline{\D u} +\Re \frac{1}{2}\int \nabla(\Delta \psi)\cdot \D u \bar{u}+ \varepsilon \Im\int \nabla \psi\cdot \overline{\D u}u \\
& - \Im  \int \sum_{j,k=1}^{d} \frac{\partial \psi}{\partial x_{k}} B_{kj}(\D)_{j} u\bar{u} -\frac{1}{2}\int \Delta\psi Q |u|^{2} -\Re\int Q\nabla\psi\cdot \D u\bar{u} \notag\\
& + \frac{\lambda}{2}\int \nabla \p \cdot \nabla\psi |u|^{2} = -\Re \int f\nabla \psi \cdot \overline{\D u} -\frac{1}{2}\Re\int f \Delta\psi\bar{u},\notag
\end{align}
where $D^{2} \psi$ denotes the Hessian of $ \psi$.
\end{lem}

\end{document}